\documentclass[11pt]{article}

\usepackage[margin=1in]{geometry}
\usepackage{amsmath,amssymb,amsthm,mathtools}
\usepackage{booktabs}
\usepackage{microtype}
\usepackage{enumitem}
\usepackage[colorlinks=true,linkcolor=blue,citecolor=blue,urlcolor=blue]{hyperref}

\newtheorem{theorem}{Theorem}[section]
\newtheorem{proposition}[theorem]{Proposition}
\newtheorem{lemma}[theorem]{Lemma}
\newtheorem{corollary}[theorem]{Corollary}
\theoremstyle{definition}
\newtheorem{definition}[theorem]{Definition}
\newtheorem{example}[theorem]{Example}
\theoremstyle{remark}
\newtheorem{remark}[theorem]{Remark}

\newcommand{\C}{\mathcal{C}}
\newcommand{\E}{\mathbb{E}}
\newcommand{\Pp}{\mathbb{P}}
\newcommand{\Thetaop}{\Theta}

\newcommand{\sig}[2]{\sigma_{#1}(#2)}
\newcommand{\StirlingII}[2]{\genfrac\{\}{0pt}{}{#1}{#2}}

\title{Initial Runs in Integer Compositions:\\Lambert Series and Generalized Divisor Sums}
\author{Igor Kleiner\\\small Department of Data Science, Holon Institute of Technology, Holon, Israel}
\date{}

\begin{document}
\maketitle

\begin{abstract}
We study the initial run of equal adjacent parts in an integer composition. For a composition $\alpha$, let $P(\alpha)$ be the common value of the parts in the initial run, let $K(\alpha)$ be the number of parts in that run, and set $S(\alpha)=P(\alpha)K(\alpha)$. We obtain exact finite formulas for the joint parameters $(P,K,S)$ and the corresponding limiting laws. In the unrestricted model, each fixed event $(P,K)=(p,k)$ reaches its limiting probability once $n$ is above an explicit threshold. The limiting distribution of the total size is
\[
\Pp(S=s)=2^{-s}\sum_{d\mid s}(1-2^{-d}).
\]
Pointing a unit cell in the initial run gives
\[
M_1(z)=\frac{1-z}{1-2z}\sum_{m\ge1}\sig{1}{m}z^m.
\]
For every $r\ge1$, the $r$th power-moment generating function has coefficient function
\[
A_r(n)=\sum_{j=1}^{r}(-1)^{j+1}\binom{r}{j}n^{r-j}\sig{j}{n},
\]
so higher moments involve explicit finite combinations of generalized divisor sums. We also give factorial-moment formulas for the initial-run length and extend the construction to compositions with parts in a prescribed set $\mathcal A$, where the corresponding formulas involve restricted divisor sums and the dominant root of the part generating function.
\end{abstract}



\section{Introduction}

An integer composition of $n$ is an ordered sequence
\[
\alpha=(a_1,\ldots,a_\ell),\qquad a_i\ge1,\qquad a_1+\cdots+a_\ell=n.
\]
We write $\C_n$ for the set of compositions of $n$. It is classical that
\[
|\C_n|=2^{n-1}\qquad(n\ge1).
\]
Compositions are among the basic sequential structures of enumerative and analytic combinatorics; see, for example, Flajolet and Sedgewick~\cite{FlajoletSedgewick2009} and Heubach and Mansour~\cite{HeubachMansour2009}.

A \emph{run} is a maximal consecutive block of equal parts. Thus the composition
\[
(3,3,3,1,4,4,2)
\]
has runs $(3,3,3)$, $(1)$, $(4,4)$, and $(2)$. Runs in integer compositions have been studied from several directions. Wilf~\cite{Wilf2011} considered restrictions on run lengths, while Gafni~\cite{Gafni2015} analyzed the longest run of equal parts in a uniformly random composition. Bender and Gao~\cite{BenderGao2016} developed a broader framework for runs of a prescribed subcomposition in locally restricted sequential structures, with applications to maximum run lengths. Closely related run statistics for words generated by independent geometrically distributed random variables were developed by Grabner, Knopfmacher, and Prodinger~\cite{GrabnerKnopfmacherProdinger2003}. The probabilistic representation of random compositions by geometric variables also underlies other composition statistics; see Hitczenko and Louchard~\cite{HitczenkoLouchard2001}. Position-sensitive statistics on random compositions include, for example, the record statistics studied by Knopfmacher and Mansour~\cite{KnopfmacherMansour2012}. A closer boundary example occurs in Shattuck~\cite{Shattuck2025}: in a capacity-statistic argument for compositions with parts in $\{1,2\}$, marked $1$'s in an initial run of $1$'s appear as an auxiliary class. That construction fixes the part value and does not study the total size of a general initial run or its moment sequence. There is also relevant OEIS work of Wiseman: A357180 records the first run-length of each composition in standard order, while A353847 and A353932 encode the run-sum transformation, replacing every equal-part run by its sum~\cite{OEISA357180,OEISA353847,OEISA353932}. Thus the raw boundary quantities $K$ and the run-sum vector were already present in the literature, although not aggregated over all compositions of a fixed size in the form considered here. Restricted-part composition models themselves form a broad classical family; see, for example, Banderier and Hitczenko~\cite{BanderierHitczenko2012} for analytic and probabilistic asymptotics and Huang~\cite{Huang2020} for enumerative identities with prescribed part restrictions.

In this paper we isolate the \emph{initial run}, the unique run containing the first part of a composition. Its position at the boundary makes it structurally different from a generic later run: it can be separated from the remaining tail without imposing a constraint on a preceding run. For the second and subsequent runs, adjacent run values must be different, so the corresponding generating functions already involve sums over several run values with inequality constraints rather than the single Lambert-series factor that appears below. Thus the initial run is the simplest boundary case of a broader run hierarchy, and it is natural to study it first.

Suppose that the first part has value $p$ and occurs $k$ consecutive times before the first change. We study the three naturally related statistics
\[
P=p,\qquad K=k,\qquad S=pk.
\]
Here $K$ is the length of the initial run in number of parts, whereas $S$ is its total size. Equivalently, after representing a part of size $p$ as $p$ unit cells, weighting a composition by $S$ counts the ways to distinguish one unit cell in its initial run.

The finite enumeration is explicit. For fixed $(p,k)$ we obtain a closed formula for the number of compositions of $n$ with $(P,K)=(p,k)$. Because $|\C_n|=2^{n-1}$, the probability of each fixed local event $(P_n,K_n)=(p,k)$, and likewise of each fixed value of $S_n$, is already equal to its limit once $n$ is large enough relative to the prescribed parameters. This finite stabilization is particular to the unrestricted model.

The arithmetic part begins with pointing. Pointing one unit cell in the initial run gives the classical Lambert series for $\sigma_1$. For general $r$, the $r$th power moment of $S$ has coefficients that are explicit binomial combinations of $\sigma_1,\ldots,\sigma_r$. The Lambert-series identities used here are classical; see Erd\H{o}s~\cite{Erdos1948} and Schmidt~\cite{Schmidt2026}. Divisor sums have also appeared previously in identities involving compositions; a particularly direct antecedent is Alegri~\cite{Alegri2022}, with further identities in Alegri~\cite{Alegri2024} and the subsequent correction~\cite{AlegriCorrection2026}. The constant $\sum_{m\ge1}\sigma_1(m)/2^m$ and a related partial-sum sequence are OEIS A066766 and A066767~\cite{OEISA066766,OEISA066767}; Section~\ref{sec:pointing} records the exact relation with our first-moment coefficients.

We begin with a joint generating function for $P$, $K$, and $S$. Exact enumeration for fixed $(P,K)$ then gives the finite laws and their limiting forms, including the distribution of $S$. Pointing yields the first Lambert-series identity, and the same decomposition gives the full sequence of power moments in terms of generalized divisor sums. Section~8 records factorial-moment forms, while Section~\ref{sec:restricted} extends the construction to prescribed part sets. Numerical checks are collected in Section~10.

The existing literature already records several ingredients of the present statistic: first-run length appears in OEIS A357180, and the run-sum transformation of a composition is recorded by A353847/A353932~\cite{OEISA357180,OEISA353847,OEISA353932}. Thus we do not claim novelty for the mere act of taking the first run or summing the entries of a run. To the best of our knowledge, however, the first run-sum has not previously been developed as an aggregate statistic over all compositions of a fixed size, with exact finite enumeration, its divisor-indexed distribution, a pointed Lambert-series interpretation, and the all-order generalized-divisor-sum moment formulas obtained below. The first-moment coefficient sequence is itself the first-difference transform of OEIS A066767, and the Lambert series and limiting constant are classical.

\section{Initial-run decomposition}

We first make the local statistics and the pointing convention precise.

\begin{definition}[Initial run and its statistics]\label{def:initialrun}
Let $\alpha=(a_1,\ldots,a_\ell)$ be a nonempty composition. Its \emph{initial run} is the unique maximal prefix
\[
(a_1,\ldots,a_k)
\]
whose entries are all equal. We write
\[
P(\alpha)=a_1,\qquad K(\alpha)=k,\qquad S(\alpha)=P(\alpha)K(\alpha).
\]
Thus $K(\alpha)$ is the number of parts in the initial run, while $S(\alpha)$ is the sum of those parts. These statistics are defined only for nonempty compositions.
\end{definition}

\begin{definition}[Unit cells and pointing]\label{def:pointing}
For the purpose of pointing, a part of size $p$ is represented by an ordered row of $p$ unit cells. Hence an initial run consisting of $k$ copies of $p$ contains exactly $pk=S(\alpha)$ unit cells. A \emph{pointed initial-run composition} is a pair $(\alpha,c)$, where $\alpha$ is a nonempty composition and $c$ is one of the unit cells belonging to the initial run of $\alpha$. Consequently, a composition $\alpha$ gives rise to exactly $S(\alpha)$ pointed objects. The unit-cell representation is auxiliary: it does not change the underlying notion of integer composition.
\end{definition}

\begin{remark}[Boundary cases]\label{rem:boundary}
If $\alpha=(n)$ has a single part, then $P(\alpha)=n$, $K(\alpha)=1$, and $S(\alpha)=n$. More generally, if all parts of $\alpha$ are equal, then the admissible tail in the decomposition below is empty. In particular, all definitions and formulas include the case $n=1$.
\end{remark}

It is convenient to include the empty composition temporarily. Let
\[
T(z)=1+\sum_{n\ge1}2^{n-1}z^n=1+\frac{z}{1-2z}=\frac{1-z}{1-2z}.
\]
Thus $T(z)$ is the ordinary generating function for all compositions, including the empty composition.

Fix $p\ge1$. A nonempty run consisting entirely of parts equal to $p$ has generating function
\[
w_p(z)=z^p+z^{2p}+z^{3p}+\cdots=\frac{z^p}{1-z^p}.
\]
After the initial run, the remaining tail is either empty or begins with a part different from $p$.

\begin{lemma}[Admissible tail]
For fixed $p\ge1$, the generating function for tails that are empty or begin with a part different from $p$ is
\[
R_p(z)=(1-z^p)T(z)=(1-z^p)\frac{1-z}{1-2z}.
\]
\end{lemma}

\begin{proof}
All possible tails, including the empty tail, are counted by $T(z)$. A nonempty tail beginning with a part equal to $p$ consists of that first part, contributing $z^p$, followed by an arbitrary composition, contributing $T(z)$. Hence the forbidden tails have generating function $z^pT(z)$, and subtraction gives
\[
R_p(z)=T(z)-z^pT(z)=(1-z^p)T(z).
\]
\end{proof}

\begin{proposition}[Canonical initial-run decomposition]\label{prop:decomp}
Every nonempty composition admits a unique decomposition
\[
\alpha=(\underbrace{p,\ldots,p}_{k\text{ copies}})\cdot\beta,
\]
where $p,k\ge1$ and $\beta$ is empty or begins with a part different from $p$.
\end{proposition}

\begin{proof}
Take $p=a_1$ and let $k$ be the largest integer for which $a_1=\cdots=a_k=p$. These choices are unique. If further parts exist, maximality gives $a_{k+1}\ne p$. Conversely, any such triple $(p,k,\beta)$ determines a unique composition.
\end{proof}

As a consistency check, Proposition~\ref{prop:decomp} implies
\[
\sum_{p\ge1}w_p(z)R_p(z)
=T(z)\sum_{p\ge1}z^p
=\frac{1-z}{1-2z}\frac{z}{1-z}
=\frac{z}{1-2z},
\]
the generating function for nonempty compositions.

\section{A master generating function}

We use the statistics $P(\alpha)$, $K(\alpha)$, and $S(\alpha)$ from Definition~\ref{def:initialrun}.

\begin{theorem}[Master generating function]\label{thm:master}
Let $x$ mark $P$, $u$ mark $K$, and $v$ mark $S$. Then
\[
\boxed{
\mathcal J(z;x,u,v)
=T(z)\sum_{p\ge1}x^p(1-z^p)
\frac{u(vz)^p}{1-u(vz)^p}.
}
\]
More explicitly,
\[
\mathcal J(z;x,u,v)
=\sum_{\alpha\ne\varnothing}
 x^{P(\alpha)}u^{K(\alpha)}v^{S(\alpha)}z^{|\alpha|}.
\]
\end{theorem}

\begin{proof}
For fixed $p$, an initial run of length $k\ge1$ contributes
\[
x^p u^k v^{pk}z^{pk}.
\]
Summing over $k$ gives
\[
x^p\sum_{k\ge1}u^k(vz)^{pk}
=x^p\frac{u(vz)^p}{1-u(vz)^p}.
\]
The admissible tail contributes $R_p(z)=(1-z^p)T(z)$. Summing over $p\ge1$ gives the result.
\end{proof}

\begin{remark}[Sanity check]
Setting $x=u=v=1$ gives
\[
\mathcal J(z;1,1,1)
=T(z)\sum_{p\ge1}z^p
=\frac{z}{1-2z},
\]
as required.
\end{remark}

\section{Exact enumeration and finite stabilization}

Define
\[
t_m=
\begin{cases}
0,&m<0,\\
1,&m=0,\\
2^{m-1},&m\ge1.
\end{cases}
\]
Thus $t_m$ counts all compositions of $m$, including the empty composition when $m=0$.

For each $n\ge1$, let $\alpha_n$ be uniformly distributed on $\C_n$, and define
\[
P_n=P(\alpha_n),\qquad K_n=K(\alpha_n),\qquad S_n=S(\alpha_n).
\]

\begin{theorem}[Exact joint enumeration]\label{thm:finitepk}
Let $N_n(p,k)$ be the number of compositions $\alpha\in\C_n$ satisfying
\[
P(\alpha)=p,\qquad K(\alpha)=k.
\]
Then
\[
\boxed{
N_n(p,k)=t_{n-pk}-t_{n-p(k+1)}.
}
\]
\end{theorem}

\begin{proof}
After the prescribed first $k$ copies of $p$, the remaining size is $n-pk$. There are $t_{n-pk}$ possible tails. Among these, the tails beginning with another $p$ are obtained by fixing that first additional $p$ and appending an arbitrary composition of the remaining size $n-p(k+1)$, giving $t_{n-p(k+1)}$ possibilities. Their subtraction enforces maximality of the initial run.
\end{proof}

\begin{corollary}[Finite stabilization of local events]\label{cor:stabilization-pk}
For fixed $p,k\ge1$, if $n>p(k+1)$, then for a uniformly chosen $\alpha_n\in\C_n$,
\[
\boxed{
\Pp(P_n=p,K_n=k)=2^{-pk}(1-2^{-p}).
}
\]
Thus, as a direct consequence of the exact geometric formula $|\C_n|=2^{n-1}$, the probability of each fixed event $(P_n,K_n)=(p,k)$ becomes exactly constant after a finite threshold.
\end{corollary}

\begin{proof}
Under the stated inequality, both $n-pk$ and $n-p(k+1)$ are positive. Hence Theorem~\ref{thm:finitepk} gives
\[
N_n(p,k)=2^{n-pk-1}-2^{n-p(k+1)-1}.
\]
Dividing by $|\C_n|=2^{n-1}$ yields the formula.
\end{proof}

\begin{example}
For $n=3$, the four compositions are
\[
(3),\quad(2,1),\quad(1,2),\quad(1,1,1).
\]
Their triples $(P,K,S)$ are respectively
\[
(3,1,3),\quad(2,1,2),\quad(1,1,1),\quad(1,3,3).
\]
The total weights are therefore
\[
\sum_{\alpha\in\C_3}S(\alpha)=9,
\qquad
\sum_{\alpha\in\C_3}K(\alpha)=6.
\]
\end{example}

\section{Limiting distributions}

Corollary~\ref{cor:stabilization-pk} immediately gives a limiting joint distribution.

\begin{corollary}[Joint limit law]\label{cor:jointlaw}
For $p,k\ge1$,
\[
\boxed{
\Pp(P=p,K=k)=2^{-pk}(1-2^{-p}).
}
\]
These probabilities sum to $1$.
\end{corollary}

\begin{proof}
The pointwise limit follows from finite stabilization. Moreover,
\[
\sum_{k\ge1}2^{-pk}(1-2^{-p})=2^{-p},
\]
and therefore
\[
\sum_{p\ge1}\sum_{k\ge1}2^{-pk}(1-2^{-p})
=\sum_{p\ge1}2^{-p}=1.
\]
\end{proof}

\begin{corollary}[Distribution of the first part]\label{cor:P}
\[
\boxed{\Pp(P=p)=2^{-p},\qquad p\ge1.}
\]
In particular,
\[
\E P=2,
\qquad
\operatorname{Var}(P)=2.
\]
\end{corollary}

\begin{corollary}[Conditional run length]\label{cor:conditionalK}
Conditionally on $P=p$, the limiting variable $K$ is geometric on $\{1,2,\ldots\}$ with parameter $1-2^{-p}$:
\[
\boxed{
\Pp(K=k\mid P=p)=(1-2^{-p})2^{-p(k-1)}.
}
\]
\end{corollary}

\begin{corollary}[Distribution of the initial-run length]\label{cor:Klaw}
For $k\ge1$,
\[
\boxed{
\Pp(K=k)=\frac{1}{2^k-1}-\frac{1}{2^{k+1}-1}.
}
\]
Equivalently,
\[
\boxed{
\Pp(K\ge k)=\frac{1}{2^k-1}.
}
\]
Consequently,
\[
\boxed{
\E K=\sum_{k\ge1}\frac{1}{2^k-1}=1.606695152415291\ldots,
}
\]
the Erd\H{o}s--Borwein constant.
\end{corollary}

\begin{proof}
Summing Corollary~\ref{cor:jointlaw} over $p$ gives
\[
\Pp(K=k)=\sum_{p\ge1}\left(2^{-pk}-2^{-p(k+1)}\right)
=\frac{1}{2^k-1}-\frac{1}{2^{k+1}-1}.
\]
The tail identity follows by telescoping. The expectation is the standard tail sum
\[
\E K=\sum_{k\ge1}\Pp(K\ge k).
\]
The constant was studied arithmetically by Erd\H{o}s~\cite{Erdos1948} and Borwein~\cite{Borwein1992}.
\end{proof}

We now turn to the total initial-run size $S=PK$.

\begin{theorem}[Distribution of the total initial-run size]\label{thm:Slaw}
For $s\ge1$,
\[
\boxed{
\Pp(S=s)=2^{-s}\sum_{d\mid s}(1-2^{-d}).
}
\]
Equivalently,
\[
\Pp(S=s)=2^{-s}\left(\tau(s)-\sum_{d\mid s}2^{-d}\right),
\]
where $\tau(s)=\sigma_0(s)$ is the number-of-divisors function.
\end{theorem}

\begin{proof}
The equality $S=PK=s$ holds exactly when $P=d$ is a positive divisor of $s$ and $K=s/d$. Substitution into Corollary~\ref{cor:jointlaw} gives
\[
\Pp(P=d,K=s/d)=2^{-s}(1-2^{-d}).
\]
Summing over $d\mid s$ proves the formula.
\end{proof}

\begin{proposition}[Exact finite law for $S_n$]\label{prop:finiteS}
For $s\ge1$,
\[
\#\{\alpha\in\C_n:S(\alpha)=s\}
=\sum_{d\mid s}\bigl(t_{n-s}-t_{n-s-d}\bigr).
\]
In particular, if $n>2s$, then
\[
\boxed{
\Pp(S_n=s)=\Pp(S=s).
}
\]
\end{proposition}

\begin{proof}
For each divisor $d\mid s$, impose $P=d$ and $K=s/d$ in Theorem~\ref{thm:finitepk}. This gives
\[
t_{n-s}-t_{n-s-d}.
\]
Summing over $d\mid s$ gives the exact count. The largest possible divisor is $d=s$, so the condition $n-s-d>0$ for every divisor reduces to $n>2s$. Under this condition, division by $2^{n-1}$ reduces each summand to
\[
2^{-s}(1-2^{-d}),
\]
which is exactly Theorem~\ref{thm:Slaw}.
\end{proof}

\section{Pointing and Lambert series}\label{sec:pointing}

By Definition~\ref{def:pointing}, $S(\alpha)$ is exactly the number of possible unit-cell pointings of the initial run of $\alpha$. For fixed $p$, the first run has generating function
\[
w_p(z)=\frac{z^p}{1-z^p}.
\]
Let
\[
\Thetaop=z\frac{d}{dz}.
\]
Then
\[
\Thetaop w_p(z)
=\frac{pz^p}{(1-z^p)^2}
=\sum_{k\ge1}pk\,z^{pk},
\]
so that the coefficient of $z^{pk}$ is $pk$, the number of unit cells available to be distinguished in a run of $k$ parts of size $p$.

For $j\ge0$, write
\[
\sigma_j(n)=\sum_{d\mid n}d^j.
\]
The classical Lambert series is
\[
D_j(z)=\sum_{d\ge1}\frac{d^jz^d}{1-z^d}
=\sum_{n\ge1}\sigma_j(n)z^n.
\]

\begin{theorem}[Pointed initial-run generating function]\label{thm:M1}
Let
\[
M_1(z)=\sum_{\alpha\ne\varnothing}S(\alpha)z^{|\alpha|}.
\]
Then
\[
\boxed{
M_1(z)=T(z)\sum_{p\ge1}\frac{pz^p}{1-z^p}
=\frac{1-z}{1-2z}\sum_{n\ge1}\sigma_1(n)z^n.
}
\]
\end{theorem}

\begin{proof}
For each $p$, point the initial run and append an admissible tail:
\[
M_1(z)=\sum_{p\ge1}\Thetaop w_p(z)R_p(z).
\]
Using the formulas above,
\[
M_1(z)
=\sum_{p\ge1}
\frac{pz^p}{(1-z^p)^2}(1-z^p)T(z)
=T(z)\sum_{p\ge1}\frac{pz^p}{1-z^p}.
\]
Finally, expanding each geometric denominator and collecting the coefficient of $z^n$ gives
\[
\sum_{p\ge1}\frac{pz^p}{1-z^p}
=\sum_{n\ge1}\left(\sum_{p\mid n}p\right)z^n
=\sum_{n\ge1}\sigma_1(n)z^n.
\]
\end{proof}

The first coefficients are
\[
M_1(z)=z+4z^2+9z^3+21z^4+41z^5+88z^6+172z^7+351z^8+\cdots.
\]

\begin{remark}[Relation with OEIS A066767]\label{rem:oeis}
Let
\[
a(n)=[z^n]M_1(z).
\]
OEIS A066767 is the sequence
\[
b(n)=\sum_{m=1}^n \sigma_1(m)2^{n-m},
\]
whose normalized values $b(n)/2^n$ are the $n$th partial sums of
$\sum_{m\ge1}\sigma_1(m)/2^m$~\cite{OEISA066767}.  Comparing this formula with
Theorem~\ref{thm:M1} gives
\[
\boxed{a(n)=b(n)-b(n-1)}
\]
with the convention $b(0)=0$.  Thus the coefficient sequence
\[
1,4,9,21,41,88,172,351,\ldots
\]
is the first-difference transform of A066767.  Theorem~\ref{thm:M1} gives this first-difference sequence a direct
combinatorial interpretation: $a(n)$ is the total number of unit-cell pointings
in initial runs of compositions of $n$.  Section~\ref{sec:power} gives the
corresponding formulas for higher moments in terms of generalized divisor sums.
\end{remark}

\begin{corollary}[Exact first moment]\label{cor:exactmeanS}
For a uniformly chosen composition of $n$,
\[
\boxed{
\E S_n
=\sum_{m=1}^{n-1}\frac{\sigma_1(m)}{2^m}
+\frac{\sigma_1(n)}{2^{n-1}}.
}
\]
\end{corollary}

\begin{proof}
Since
\[
T(z)=1+\sum_{j\ge1}2^{j-1}z^j,
\]
Theorem~\ref{thm:M1} gives
\[
[z^n]M_1(z)=\sigma_1(n)+\sum_{m=1}^{n-1}2^{n-m-1}\sigma_1(m).
\]
Divide by $2^{n-1}$.
\end{proof}

\begin{corollary}[Limiting first moment]\label{cor:limitmeanS}
\[
\boxed{
\lim_{n\to\infty}\E S_n
=\sum_{m\ge1}\frac{\sigma_1(m)}{2^m}
=\sum_{p\ge1}\frac{p}{2^p-1}
=2.744033888759488\ldots.
}
\]
\end{corollary}

\begin{remark}[The limiting constant]
The constant
\[
\sum_{m\ge1}\frac{\sigma_1(m)}{2^m}
=\sum_{p\ge1}\frac{p}{2^p-1}
=2.74403388875948836048\ldots
\]
predates the present interpretation and is catalogued as OEIS A066766~\cite{OEISA066766}.
Here the same constant appears as the limiting expected total size of the
initial run.  The exact finite formulas and the higher-moment identities below
give the corresponding composition interpretation.
\end{remark}

\begin{proof}
The exact formula in Corollary~\ref{cor:exactmeanS} converges because $\sigma_1(m)$ grows at most polynomially, while $2^{-m}$ decays exponentially. The equivalence of the two infinite sums follows from the Lambert-series identity evaluated at $z=1/2$.
\end{proof}

\begin{remark}[Analytic viewpoint]
Theorem~\ref{thm:M1} also gives a one-line singularity argument. The Lambert series $D_1(z)$ is analytic in $|z|<1$, while $T(z)$ has a simple pole at $z=1/2$. Thus the dominant singularity of $M_1$ is inherited from $T$, and the limiting mean is $D_1(1/2)$. The exact coefficient formula above is stronger and requires no asymptotic machinery.
\end{remark}

\section{All power moments and generalized divisor sums}\label{sec:power}

For $r\ge1$, define
\[
M_r(z)=\sum_{\alpha\ne\varnothing}S(\alpha)^r z^{|\alpha|}.
\]
Repeated application of $\Thetaop$ gives
\[
\Thetaop^r w_p(z)=\sum_{k\ge1}(pk)^r z^{pk}.
\]

\begin{theorem}[Power-moment divisor-sum hierarchy]\label{thm:powermoments}
For every integer $r\ge1$,
\[
\boxed{
M_r(z)=T(z)\sum_{n\ge1}A_r(n)z^n,
}
\]
where
\[
\boxed{
A_r(n)=\sum_{d\mid n}\bigl(n^r-(n-d)^r\bigr).
}
\]
Equivalently,
\[
\boxed{
A_r(n)=\sum_{j=1}^{r}(-1)^{j+1}\binom{r}{j}n^{r-j}\sigma_j(n).
}
\]
\end{theorem}

\begin{proof}
By the initial-run decomposition,
\[
M_r(z)
=\sum_{p\ge1}R_p(z)\Thetaop^r w_p(z)
=T(z)\sum_{p\ge1}(1-z^p)\sum_{k\ge1}(pk)^r z^{pk}.
\]
Consider the coefficient of $z^n$ in the inner sum. The positive contribution occurs whenever $p\mid n$ and equals $n^r$. The shifted negative contribution also occurs for $p\mid n$ and equals $(n-p)^r$; when $p=n$ this is $0$, so the same formula remains valid. Hence
\[
A_r(n)=\sum_{p\mid n}\bigl(n^r-(n-p)^r\bigr).
\]
Expanding the difference by the binomial theorem gives
\[
n^r-(n-p)^r
=\sum_{j=1}^r(-1)^{j+1}\binom{r}{j}n^{r-j}p^j.
\]
Summation over $p\mid n$ yields the second formula.
\end{proof}

The first four arithmetic coefficient functions are
\begin{align*}
A_1(n)&=\sigma_1(n),\\
A_2(n)&=2n\sigma_1(n)-\sigma_2(n),\\
A_3(n)&=3n^2\sigma_1(n)-3n\sigma_2(n)+\sigma_3(n),\\
A_4(n)&=4n^3\sigma_1(n)-6n^2\sigma_2(n)+4n\sigma_3(n)-\sigma_4(n).
\end{align*}

\begin{corollary}[Lambert-series derivative form]\label{cor:thetaD}
Let
\[
D_j(z)=\sum_{n\ge1}\sigma_j(n)z^n
=\sum_{d\ge1}\frac{d^jz^d}{1-z^d}.
\]
Then
\[
\boxed{
M_r(z)
=T(z)\sum_{j=1}^{r}(-1)^{j+1}\binom{r}{j}\Thetaop^{\,r-j}D_j(z).
}
\]
\end{corollary}

\begin{proof}
The coefficient of $z^n$ in $\Thetaop^{r-j}D_j(z)$ is $n^{r-j}\sigma_j(n)$. Apply Theorem~\ref{thm:powermoments}.
\end{proof}

\begin{corollary}[Exact and limiting power moments]\label{cor:powerlimit}
For $n,r\ge1$,
\[
\boxed{
\E[S_n^r]
=\sum_{m=1}^{n-1}\frac{A_r(m)}{2^m}
+\frac{A_r(n)}{2^{n-1}}.
}
\]
Consequently,
\[
\boxed{
\lim_{n\to\infty}\E[S_n^r]
=\sum_{m\ge1}\frac{A_r(m)}{2^m}.
}
\]
\end{corollary}

\begin{proof}
The coefficient extraction is identical to Corollary~\ref{cor:exactmeanS}. The limiting series converges absolutely because, for fixed $r$, $A_r(m)$ grows at most polynomially.
\end{proof}

For reference,
\begin{align*}
\E S&=2.744033888759488\ldots,\\
\E S^2&=10.576850962011492\ldots,\\
\E S^3&=54.176877111337226\ldots,\\
\E S^4&=349.962961516483268\ldots.
\end{align*}
Thus
\[
\boxed{\operatorname{Var}(S)=3.047128979350972\ldots.}
\]

\section{Factorial moments and Stirling transforms}

We next record the factorial-moment forms of these results.  For the total size $S$, factorial moments are equivalent to the power moments by Stirling inversion.  For the run length $K$, they give a separate compact Lambert-type formula.

For $r\ge1$, write
\[
(x)_r=x(x-1)\cdots(x-r+1)
\]
for the falling factorial, and define
\[
F_r(z)=\sum_{\alpha\ne\varnothing}(S(\alpha))_r z^{|\alpha|}.
\]

\begin{proposition}[Factorial moments of the total size]\label{prop:Sfactorial}
For every $r\ge1$,
\[
\boxed{
F_r(z)=T(z)\sum_{n\ge1}C_r(n)z^n,
\qquad
C_r(n)=\sum_{d\mid n}\bigl((n)_r-(n-d)_r\bigr).
}
\]
Moreover, if $\StirlingII{r}{j}$ and $s(r,j)$ denote Stirling numbers of the second kind and signed Stirling numbers of the first kind, then
\[
\boxed{
M_r(z)=\sum_{j=1}^{r}\StirlingII{r}{j}F_j(z),
\qquad
F_r(z)=\sum_{j=1}^{r}s(r,j)M_j(z).
}
\]
The same identities hold after coefficient extraction and for the corresponding limiting moments.
\end{proposition}

\begin{proof}
Differentiate the master generating function $r$ times with respect to $v$ and set $x=u=v=1$.  This replaces the weight $(pk)^r$ by $(pk)_r$ and gives
\[
C_r(n)=\sum_{d\mid n}\bigl((n)_r-(n-d)_r\bigr).
\]
The two transform identities are the classical relations
\[
x^r=\sum_{j=0}^{r}\StirlingII{r}{j}(x)_j,
\qquad
(x)_r=\sum_{j=0}^{r}s(r,j)x^j,
\]
applied to $x=S(\alpha)$ and summed over all nonempty compositions.
\end{proof}

Factorial moments of $S$ therefore contain the same information as the power moments, so we do not list their low-order expansions.  For the run length $K$, the corresponding formula takes a simpler form.  Define
\[
G_r(z)=\sum_{\alpha\ne\varnothing}(K(\alpha))_r z^{|\alpha|}.
\]

\begin{proposition}[Factorial moments of the run length]\label{prop:Kfactorial}
For $r\ge1$,
\[
\boxed{
G_r(z)=r!\,T(z)\sum_{p\ge1}\frac{z^{pr}}{(1-z^p)^r}.
}
\]
If
\[
G_r(z)=T(z)\sum_{n\ge1}B_r(n)z^n,
\]
then
\[
\boxed{
B_r(n)=r!\sum_{d\mid n}\binom{d-1}{r-1}.
}
\]
Consequently,
\[
\boxed{
\lim_{n\to\infty}\E[(K_n)_r]
=r!\sum_{p\ge1}\frac{1}{(2^p-1)^r}.
}
\]
\end{proposition}

\begin{proof}
Differentiate Theorem~\ref{thm:master} $r$ times with respect to $u$ and set $x=v=u=1$. Since
\[
\sum_{k\ge1}(k)_r x^k=\frac{r!x^r}{(1-x)^{r+1}},
\]
the tail factor $1-z^p$ gives
\[
(1-z^p)\sum_{k\ge1}(k)_r z^{pk}
=\frac{r!z^{pr}}{(1-z^p)^r}.
\]
Expanding
\[
\frac{z^{pr}}{(1-z^p)^r}
=\sum_{d\ge r}\binom{d-1}{r-1}z^{pd}
\]
and collecting the coefficient of $z^n$ yields the formula for $B_r(n)$.  The limiting factorial moment follows by the same coefficient argument as in Corollary~\ref{cor:powerlimit}, equivalently by evaluating the analytic factor at $z=1/2$.
\end{proof}

For $r=1$ this reduces to
\[
B_1(n)=\sigma_0(n)
\quad\text{and}\quad
\lim_{n\to\infty}\E K_n
=\sum_{p\ge1}\frac{1}{2^p-1},
\]
recovering the Erd\H{o}s--Borwein constant from Corollary~\ref{cor:Klaw}.

\section{Extension to prescribed part sets}\label{sec:restricted}

The preceding decomposition is not specific to unrestricted compositions.  Let
$\mathcal A\subseteq\mathbb Z_{>0}$ be a nonempty set of allowed part sizes and write
\[
B_{\mathcal A}(z)=\sum_{a\in\mathcal A}z^a,
\qquad
T_{\mathcal A}(z)=\frac{1}{1-B_{\mathcal A}(z)}.
\]
Thus $T_{\mathcal A}$ is the ordinary generating function for compositions whose
parts lie in $\mathcal A$, including the empty composition.  We write
\[
t_{\mathcal A}(n)=[z^n]T_{\mathcal A}(z)
\]
with the convention $t_{\mathcal A}(n)=0$ for $n<0$.
For $j\geq0$ define the restricted divisor sums
\[
\sigma_{j,\mathcal A}(n)
   =\sum_{\substack{d\mid n\\ d\in\mathcal A}} d^j.
\]
When $\mathcal A=\mathbb Z_{>0}$ these reduce to the usual generalized divisor sums.

\begin{proposition}[Restricted-part first-run theory]\label{prop:restricted}
For compositions with all parts in $\mathcal A$, the following statements hold.

\begin{enumerate}[label=(\alph*)]
\item The joint generating function marking the first-part value $P$, initial-run
length $K$, and total initial-run size $S$ is
\[
\boxed{
\mathcal J_{\mathcal A}(z;x,u,v)
=
T_{\mathcal A}(z)
\sum_{p\in\mathcal A}
 x^p(1-z^p)\frac{u(vz)^p}{1-u(vz)^p}.}
\]

\item The number $N^{\mathcal A}_n(p,k)$ of such compositions of $n$ with
$(P,K)=(p,k)$ is
\[
\boxed{
N^{\mathcal A}_n(p,k)
=t_{\mathcal A}(n-pk)-t_{\mathcal A}(n-p(k+1)),
\qquad p\in\mathcal A.}
\]

\item If
\[
M_{r,\mathcal A}(z)
 =\sum_{\alpha}S(\alpha)^r z^{|\alpha|},
\qquad r\geq1,
\]
where the sum is over nonempty $\mathcal A$-restricted compositions, then
\[
\boxed{
M_{r,\mathcal A}(z)
=T_{\mathcal A}(z)
\sum_{n\geq1}A_{r,\mathcal A}(n)z^n,}
\]
where
\[
\boxed{
A_{r,\mathcal A}(n)
=
\sum_{\substack{d\mid n\\d\in\mathcal A}}
\bigl(n^r-(n-d)^r\bigr)
=
\sum_{j=1}^{r}(-1)^{j+1}\binom rj
 n^{r-j}\sigma_{j,\mathcal A}(n).}
\]
In particular,
\[
M_{1,\mathcal A}(z)
=T_{\mathcal A}(z)
\sum_{n\geq1}\sigma_{1,\mathcal A}(n)z^n.
\]
\end{enumerate}
\end{proposition}

\begin{proof}
For fixed $p\in\mathcal A$, an initial run of copies of $p$ is followed either by
an empty tail or by a composition whose first part is different from $p$.  Since a
nonempty tail beginning with $p$ has generating function $z^pT_{\mathcal A}(z)$,
the admissible tail has generating function
\[
T_{\mathcal A}(z)-z^pT_{\mathcal A}(z)
=(1-z^p)T_{\mathcal A}(z).
\]
Multiplying by the marked initial-run series and summing over $p\in\mathcal A$
gives part (a).  Extracting the coefficient corresponding to a fixed run $p^k$
gives part (b).

For the $r$th power moment, the contribution of a run $p^k$ is $(pk)^r$.  Hence
\[
M_{r,\mathcal A}(z)
=T_{\mathcal A}(z)
\sum_{p\in\mathcal A}(1-z^p)
\sum_{k\geq1}(pk)^r z^{pk}.
\]
For a fixed exponent $n$, the first term contributes $n^r$ for each divisor
$p\mid n$ belonging to $\mathcal A$, while multiplication by $z^p$ contributes
$(n-p)^r$.  This yields the first formula for $A_{r,\mathcal A}(n)$, and the
binomial theorem gives the second.
\end{proof}

The exact finite stabilization of the unrestricted model is exceptional: it comes
from the exact geometric coefficient formula $2^{n-1}$.  For a general allowed
part set, the same local law survives asymptotically under the standard
supercritical sequence conditions.

\begin{proposition}[Root-governed limiting law]\label{prop:restrictedlimit}
Assume that $B_{\mathcal A}(z)$ is analytic in a disk of radius strictly larger
than some $\rho\in(0,1)$, that
\[
B_{\mathcal A}(\rho)=1,
\]
and that $\gcd(\mathcal A)=1$.  Then $\rho$ is the unique dominant zero of
$1-B_{\mathcal A}(z)$ and
\[
t_{\mathcal A}(n)
\sim
\frac{\rho^{-n-1}}{B'_{\mathcal A}(\rho)}.
\]
Consequently, for fixed $p\in\mathcal A$ and $k\geq1$,
\[
\boxed{
\Pp(P_n=p,K_n=k)
\longrightarrow
\rho^{pk}(1-\rho^p).}
\]
The limiting marginals satisfy
\[
\Pp(P=p)=\rho^p,
\qquad
\Pp(K\geq k)=B_{\mathcal A}(\rho^k),
\]
and the limiting total-size distribution is
\[
\boxed{
\Pp(S=s)
=
\rho^s
\sum_{\substack{d\mid s\\d\in\mathcal A}}
(1-\rho^d).}
\]
Moreover, for each fixed $r\geq1$,
\[
\boxed{
\lim_{n\to\infty}\E[S_n^r]
=
\sum_{m\geq1}A_{r,\mathcal A}(m)\rho^m,}
\]
whenever the displayed series is finite.
\end{proposition}

\begin{proof}
The stated hypotheses give the standard supercritical sequence schema: the
simple pole of $T_{\mathcal A}(z)=1/(1-B_{\mathcal A}(z))$ at $z=\rho$ is the
unique dominant singularity.  Since $B'_{\mathcal A}(\rho)>0$,
\[
T_{\mathcal A}(z)
\sim
\frac{1}{\rho B'_{\mathcal A}(\rho)}
\frac{1}{1-z/\rho},
\]
which gives the coefficient estimate and, for each fixed $m$,
\[
\frac{t_{\mathcal A}(n-m)}{t_{\mathcal A}(n)}\longrightarrow\rho^m.
\]
Apply this ratio to Proposition~\ref{prop:restricted}(b) to obtain the joint
limit.  Summation over $k$ gives
\[
\Pp(P=p)=\sum_{k\geq1}\rho^{pk}(1-\rho^p)=\rho^p,
\]
and these probabilities sum to one because $B_{\mathcal A}(\rho)=1$.
Summing over $p$ gives
\[
\Pp(K\geq k)=\sum_{p\in\mathcal A}\rho^{pk}
=B_{\mathcal A}(\rho^k).
\]
The formula for $S$ follows by setting $k=s/d$ and summing over admissible
divisors $d\mid s$.  Finally, let
\[
H_{r,\mathcal A}(z)=\sum_{m\geq1}A_{r,\mathcal A}(m)z^m.
\]
For each fixed $r$, Proposition~\ref{prop:restricted}(c) and the divisor
representation imply
\[
|A_{r,\mathcal A}(m)|
\leq \sum_{d\mid m}
\bigl(m^r-(m-d)^r\bigr)
=O\!\left(m^r\tau(m)\right).
\]
Hence $H_{r,\mathcal A}(z)$ has radius of convergence at least $1$ and is
analytic in a neighborhood of $\rho<1$.  The $r$th moment generating
function is $T_{\mathcal A}(z)H_{r,\mathcal A}(z)$, so the same simple-pole
coefficient ratio gives the limiting moment $H_{r,\mathcal A}(\rho)$,
which is the displayed series.
\end{proof}

\begin{example}[Odd parts]\label{ex:oddparts}
Let $\mathcal A=\{1,3,5,\ldots\}$.  Then
\[
B_{\mathcal A}(z)=\frac{z}{1-z^2},
\qquad
\rho=\frac{\sqrt5-1}{2}.
\]
The moment coefficients are expressed through sums over odd divisors,
\[
\sigma_{j,\mathrm{odd}}(n)
=\sum_{\substack{d\mid n\\d\text{ odd}}}d^j,
\]
and, for example,
\[
M_{1,\mathrm{odd}}(z)
=
\frac{1}{1-z/(1-z^2)}
\sum_{n\geq1}\sigma_{1,\mathrm{odd}}(n)z^n.
\]
Hence the same formulas involve divisor sums restricted to odd divisors.
\end{example}

\section{Numerical checks and initial coefficients}

The following table records direct brute-force sums over all compositions for small $n$. It provides an elementary check independent of the generating-function derivations.

\begin{center}
\begin{tabular}{rrrrrr}
\toprule
$n$ & $|\C_n|$ & $\sum S$ & $\sum K$ & $\E S_n$ & $\E K_n$\\
\midrule
1 & 1   & 1   & 1   & 1.000000 & 1.000000\\
2 & 2   & 4   & 3   & 2.000000 & 1.500000\\
3 & 4   & 9   & 6   & 2.250000 & 1.500000\\
4 & 8   & 21  & 13  & 2.625000 & 1.625000\\
5 & 16  & 41  & 25  & 2.562500 & 1.562500\\
6 & 32  & 88  & 52  & 2.750000 & 1.625000\\
7 & 64  & 172 & 102 & 2.687500 & 1.593750\\
8 & 128 & 351 & 206 & 2.742188 & 1.609375\\
\bottomrule
\end{tabular}
\end{center}

The first coefficient sequences for the power-moment generating functions are
\begin{align*}
[z^n]M_1(z)&:\quad 1,4,9,21,41,88,172,351,700,1405,\ldots,\\
[z^n]M_2(z)&:\quad 1,8,23,67,133,326,620,1333,2654,5395,\ldots,\\
[z^n]M_3(z)&:\quad 1,16,63,237,503,1468,2758,6471,12976,27127,\ldots,\\
[z^n]M_4(z)&:\quad 1,32,179,883,2089,7406,14096,37117,76094,166615,\ldots.
\end{align*}
A short verification program accompanying this manuscript enumerates all compositions up to a user-specified bound and checks Theorems~\ref{thm:finitepk}, \ref{thm:M1}, \ref{thm:powermoments}, and \ref{prop:Kfactorial} directly.

\section{Concluding remarks}

We considered three parameters of the initial run: its part value $P$, its length $K$, and its total size $S=PK$.  The joint generating function gives exact finite enumeration and the limiting distributions of these parameters.

For the total size, pointing gives the Lambert series for $\sigma_1$.  More generally, the $r$th power moment is described by a finite binomial combination of $\sigma_1,\ldots,\sigma_r$.  The run length has the limiting mean equal to the Erd\H{o}s--Borwein constant, and its higher factorial moments are
\[
r!\sum_{p\ge1}(2^p-1)^{-r}.
\]
The unrestricted model also has the finite-stabilization property derived in Section~4; this follows from the exact formula $|\C_n|=2^{n-1}$.

For a prescribed set of allowed parts $\mathcal A$, the same initial-run decomposition leads to restricted divisor sums and a limiting law determined by the dominant root of $B_{\mathcal A}(z)=1$.  Exact finite stabilization need not persist in that setting.  Another natural direction is to study later runs. The first run is a boundary case: for the second and subsequent runs one must also keep track of the preceding run value and impose inequality between adjacent run values, so the one-factor Lambert-series structure is replaced by coupled sums with diagonal corrections. We leave this later-run hierarchy for future work. It would also be natural to ask which parts of the argument survive under adjacency restrictions, such as Carlitz compositions, or under other local restrictions.

\section*{Data and code availability}
All formulas in the paper are exact. A compact verification script that enumerates integer compositions and checks the initial coefficients and moment identities is included as ancillary material with the arXiv submission.

\end{document}